\documentclass[12pt]{amsart}

\usepackage{amssymb, hyperref}

\usepackage[all]{xy}

\theoremstyle{plain}

\newtheorem{theorem}{Theorem}
\newtheorem{lemma}[theorem]{Lemma}

\newtheorem{corollary}[theorem]{Corollary}

\theoremstyle{definition}

\newtheorem{remark}[theorem]{Remark}

\newcommand\bC{{\mathbb C}}

\newcommand\bG{{\mathbb G}}

\newcommand\bP{{\mathbb P}}

\newcommand\bZ{{\mathbb Z}}

\newcommand\cO{{\mathcal O}}

\newcommand\fg{\mathfrak{g}}
\newcommand\fh{\mathfrak{h}}

\newcommand\tG{{\widetilde G}}

\newcommand\tK{{\widetilde K}}

\newcommand\aff{{\rm aff}}

\newcommand\charac{{\rm char}}

\renewcommand\div{{\rm div}}

\newcommand\id{{\rm id}}

\newcommand\Alb{{\rm Alb}}
\newcommand\Aut{{\rm Aut}}

\newcommand\Der{{\rm Der}}

\newcommand\Grass{{\rm Grass}}

\newcommand\PGL{{\rm PGL}}

\newcommand\Spec{{\rm Spec}}
\newcommand\Stab{{\rm Stab}}

\newcommand\Supp{{\rm Supp}}

\numberwithin{equation}{section}

\title{Some automorphism groups are linear algebraic}

\author{Michel Brion}

\date{}

\begin{document}

\begin{abstract}
Consider a normal projective variety $X$, a linear algebraic
subgroup $G$ of $\Aut(X)$, and the field $K$ of $G$-invariant
rational functions on $X$. We show that the subgroup of $\Aut(X)$
that fixes $K$ pointwise is linear algebraic. If $K$ has 
transcendence degree $1$ over $k$, then $\Aut(X)$ is an algebraic 
group. 
\end{abstract}

\maketitle

\section{Introduction}
\label{sec:int}

Let $X$ be a projective variety over an algebraically closed field
$k$. The automorphism group functor of $X$, which assigns to any
$k$-scheme $S$ the group of relative automorphisms $\Aut_S(X \times S)$,
is represented by a $k$-group scheme $\Aut_X$, locally of finite 
type (see \cite[p.~268]{Gro61}). Thus, the reduced subscheme of
$\Aut_X$ is equipped with a structure of a smooth $k$-group scheme
$\Aut(X)$. But $\Aut(X)$ is not necessarily an algebraic group; 
equivalently, it may have infinitely many components,
e.g. when $X$ is a product of two isogenous elliptic curves. 
Still, the automorphism group is known to be a linear algebraic 
group for some interesting classes of varieties including smooth 
Fano varieties, complex almost homogeneous manifolds 
(this is due to Fu and Zhang, see \cite[Thm.~1.2]{FZ}) 
and normal almost homogeneous varieties in arbitrary 
characteristics (see \cite[Thm.~1]{Br19}; we say that $X$ 
is almost homogeneous if it admits an action of a smooth 
connected algebraic group with an open dense orbit). 
In this note, we obtain a relative version of the latter result:

\begin{theorem}\label{thm:rel}
Let $X$ be a normal projective variety, $G \subset \Aut(X)$ 
a linear algebraic subgroup, $K := k(X)^G$ the field of 
$G$-invariant rational functions on $X$, and $\Aut_K(X)$ 
the subgroup of $\Aut(X)$ fixing $K$ pointwise. 
Then $\Aut_K(X)$ is a linear algebraic group.
\end{theorem}

Here and later, closed subgroups of $\Aut(X)$ are
equipped with their reduced subscheme structure unless
otherwise specified. Thus, they may be identified with 
their group of $k$-rational points. 

By a theorem of Rosenlicht (see 
\cite{Rosenlicht}, and \cite{BGR} for a modern proof),
the rational functions in $K$ separate the $G$-orbit closures 
of general points of $X$. Thus, $\Aut_K(X)$ is 
the largest subgroup of $\Aut(X)$ having the same general 
orbit closures as $G$. Also, note that $K = k$ if and only
if $X$ is almost homogeneous; then $\Aut_K(X)$ is just
the full automorphism group.

The proof of Theorem \ref{thm:rel} is presented in 
Section \ref{sec:rel}. As in \cite{FZ,Br19}, the idea
is to construct a big line bundle on $X$, invariant under
$\Aut_K(X)$. To handle our relative situation, we use 
some tools from algebraic geometry over an arbitrary 
field, which seem to be unusual in this setting.

As an application of Theorem \ref{thm:rel} and its proof,
we show that $\Aut(X)$ is a linear algebraic group under
additional assumptions:

\begin{theorem}\label{thm:full}
Let $X$ be a normal projective variety, $G \subset \Aut(X)$
a linear algebraic subgroup, and $K := k(X)^G$. If $K$ has
transcendence degree $1$ over $k$, then $\Aut(X)$ is 
an algebraic group. If in addition $G$ is connected and
$K$ is not the function field of an elliptic curve, then
$\Aut(X)$ is linear.
\end{theorem}

This is again due (in essence) to Fu and Zhang when 
$X$ is a complex manifold and $K \simeq \bC(t)$, see 
\cite[Appl.~1.4]{FZ}. Another instance where this result was known 
occurs when $G$ is a torus, say $T$. Then $X$ is said to be 
a $T$-variety of complexity one; its automorphism group 
is explicitly described in \cite{AHHL}, when $K \simeq \bC(t)$
again.
 
Theorem \ref{thm:full} is proved in Section \ref{sec:full},
as well as the following consequence:

\begin{corollary}\label{cor:surface}
Let $X$ be a normal projective surface having 
a non-trivial action of a connected linear algebraic
group. Then $\Aut(X)$ is an algebraic group. 
Moreover, $\Aut(X)$ is linear unless $X$ is
birationally equivalent to $Y \times \bP^1$ 
for some elliptic curve $Y$.
\end{corollary}

For smooth rational surfaces, this result is due to Harbourne
(see \cite[Cor.~1.4]{Harbourne}); for ruled surfaces over
an elliptic curve, it also follows from the explicit description
of automorphism groups obtained by Maruyama 
(see \cite[Thm.~3]{Maruyama}). Note that the linearity
assumption for the acting group cannot be suppressed, as shown
again by the example of a product of two isogenous elliptic
curves. Also, there exist (much more elaborate) examples of
smooth projective surfaces having a discrete, non-finitely 
generated automorphism group (see \cite{DO,Oguiso}). 

Theorem \ref{thm:rel} may be viewed as a first step 
towards ``Galois theory'' for linear algebraic subgroups
of $\Aut(X)$. With this in mind, it would be interesting 
to characterize the fields of invariants of linear algebraic groups
among the subfields $K \subset k(X)$, and the linear algebraic 
groups which occur as $\Aut_K(X)$. 
In characteristic $0$, it is easy to see that a subfield 
$K \subset k(X)$ is the field of invariants of a connected 
linear algebraic group if and only if $K$ is algebraically
closed in $k(X)$, the $k(X)$-vector space 
$\Der_K(k(X))$ is spanned by global vector fields, and 
$X$ is unirational over $\bar{K}$.

\medskip

\noindent
{\bf Notation and conventions.}
The ground field $k$ is algebraically closed, of arbitrary 
characteristic.
A \emph{variety} is an integral separated scheme of finite type.
An \emph{algebraic group} $G$ is a group scheme of finite type.
The \emph{neutral component} $G^0$ is the connected
component of $G$ containing the neutral element.
We say that $G$ is \emph{linear} if it is smooth and affine.

\section{Proof of Theorem \ref{thm:rel}}
\label{sec:rel}

It proceeds via a sequence of reduction steps and lemmas.
We begin with two easy and useful observations:

\begin{lemma}\label{lem:closed}

\begin{enumerate}

\item[{\rm (i)}] $\Aut_K(X)$ is a closed subgroup of $\Aut(X)$.

\item[{\rm (ii)}] If $G$ is connected, then $K$ is algebraically
closed in $k(X)$.

\end{enumerate}

\end{lemma}

\begin{proof}
(i) It suffices to show that the stabilizer $\Aut_f(X)$ is closed in 
$\Aut(X)$ for any non-zero $f \in k(X)$. Let
\[ \Aut_{(f)}(X) := \{ g \in \Aut(X) ~\vert~ g^*(f) = \lambda(g) f
\text{ for some } \lambda(g) \in k^* \}. \]
Then $\Aut_f(X) \subset \Aut_{(f)}(X) \subset \Aut(X)$. 

Denote by $D_0$ (resp.~$D_{\infty}$) the scheme of zeroes 
(resp.~poles) of $f$, and by $\Aut(X,D_0,D_{\infty}) \subset \Aut(X)$ 
the common stabilizer of these subschemes of $X$. 
We claim that $\Aut_{(f)}(X) = \Aut(X,D_0,D_{\infty})$. Indeed, 
the inclusion $\Aut_{(f)}(X) \subset \Aut(X,D_0,D_{\infty})$ 
is obvious, and the opposite inclusion follows from the 
fact that every $g \in \Aut(X,D_0,D_{\infty})$ satisfies
$\div(g^*(f)) = \div(f)$.

By the claim, $\Aut_{(f)}(X)$ is closed in $\Aut(X)$. Moreover,
$\Aut_{(f)}(X)$ stabilizes the open subset 
$U := X \setminus (D_0 \cup D_{\infty}) \subset X$, and
$\Aut_f(X)$ is the stabilizer of $f \in \cO(U)$. So 
$\Aut_f(X)$ is closed in $\Aut_{(f)}(X)$.

(ii) Let $f \in k(X)$ be algebraic over $K$. Then the stabilizer 
of $f$ in $G$ is a closed subgroup of finite index, and hence 
is the whole $G$. So $f \in K$.

\end{proof}

\medskip \noindent
{\bf Step 1.} We may assume that $G$ is connected.

\begin{proof}
Denote by $\pi_0(G) := G/G^0$ the group of components.
Then the invariant field $L := k(X)^{G^0}$ is equipped
with an action of the finite group $\pi_0(G)$, and
$L^{\pi_0(G)} = K$. Thus, $L/K$ is a Galois extension with
Galois group a quotient of $\pi_0(G)$. As $G^0$ is connected, 
$L$ is algebraically closed in $k(X)$ (Lemma \ref{lem:closed}). 
Therefore, $L$ is the algebraic closure of $K$ in $k(X)$, 
and hence is stable under $\Aut_K(X)$. Since the composition 
$G \to \Aut_K(X) \to \Aut_K(L)$ 
is surjective and sends $G^0$ to the identity, this yields 
an exact sequence of algebraic groups
\[ 1 \longrightarrow \Aut_L(X) \longrightarrow \Aut_K(X)
\longrightarrow \Aut_K(L)  \longrightarrow 1, \] 
and in turn the assertion in view of Lemma \ref{lem:closed}
again.
\end{proof}

\begin{lemma}\label{lem:neutral}
$\Aut_K(X)^0$ is a linear algebraic group.
\end{lemma}

\begin{proof}
Let $\tG := \Aut_K(X)^0$. Then $G \subset \tG$ and
$k(X)^{\tG} = K$. In view of Rosenlicht's theorem mentioned
in the introduction, there exists a dense open $G$-stable 
subset $U \subset X$ such that the orbit $G \cdot x$ is open 
in $\tG \cdot x$ for all $x \in U$. Since $G$ is connected 
and linear, it follows that the variety $\tG \cdot x$ is unirational, 
and hence its Albanese variety $\Alb(\tG \cdot x)$ is trivial. 

We now recall a group-theoretic description of 
$\Alb(\tG \cdot x)$. By Chevalley's structure theorem,
$\tG$ has a largest connected linear normal  subgroup 
$\tG_{\aff}$ and the quotient $\tG/\tG_{\aff}$ is an abelian 
variety (see e.g. \cite[Thm.~1.1.1]{BSU}). Denote by 
$H = \tG_x$ the stabilizer of $x$; then 
$\tG/(\tG_{\aff} \, H)$ is an abelian variety as well. 
As the Albanese morphism of $\tG$ (resp.~of
$\tG \cdot x \simeq \tG/H$) is invariant by $\tG_{\aff}$, 
it follows readily that $\Alb(\tG) = \tG/\tG_{\aff}$ and 
$\Alb(\tG \cdot x) = \tG/(\tG_{\aff} \, H)$. Thus, 
$\tG = \tG_{\aff} \, H$. But since $H$ is affine
(see e.g. \cite[Lem.~2.3.2]{BSU}), the natural map 
$\tG/\tG_{\aff} \to \tG/\tG_{\aff} \, H$ is an isogeny.
Therefore, $\tG = \tG_{\aff}$.
\end{proof}

\medskip \noindent
{\bf Step 2.} In view of Step 1 and Lemma \ref{lem:neutral}, 
we may assume that $G = \Aut_K(X)^0$. 
Then $K$ is algebraically closed in $k(X)$ by Lemma
\ref{lem:closed}.

We may further assume that there exists a morphism 
$f : X \to Y$, where $Y$ is a normal projective variety 
satisfying the following conditions: 
$f$ induces an isomorphism $k(Y) \simeq K$, 
we have $f_*(\cO_X) = \cO_Y$, and $\Aut_K(X)$ is isomorphic
to a closed subgroup of the relative automorphism group
$\Aut_Y(X)$, containing $\Aut_Y(X)^0$. 

\begin{proof}
Choose generators $f_1,\ldots,f_n$ of the field extension
$K/k$. This defines a rational map 
$(f_1,\ldots,f_n) : X \dasharrow \bP^{n-1}$ and hence 
a rational map $f : X \dasharrow Y$, where $Y$ is a normal
projective variety and $f$ induces an isomorphism 
$k(Y) \simeq K$.
The normalization of the graph of $f$ yields a normal
projective variety $X'$ equipped with morphisms
\[ f' : X' \longrightarrow Y, \quad g : X' \longrightarrow X \] 
such that $g$ is birational and $f' = f \circ g$. Thus, 
$f'$ also induces an isomorphism $k(Y) \simeq K$.
Consider the Stein factorization of $f'$ as
$X' \stackrel{\varphi}{\longrightarrow} Z
\stackrel{\psi}{\longrightarrow} Y$.
Then $\psi$ is finite, and hence so is the extension 
$k(Z)/k(Y)$. But $k(Z) \subset k(X') \simeq k(X)$,
and $k(Y) \simeq K$ is algebraically closed in $k(X)$.
Therefore, $\psi$ is birational. As $Y$ is normal,
$\psi$ is an isomorphism by Zariski's Main Theorem.
Thus, $f_*(\cO_X) = \cO_Y$.

By construction, $\Aut_K(X)$ acts on $X'$ and $f'$ is invariant 
under this action. This yields a homomorphism 
\[ u : \Aut_K(X) \longrightarrow \Aut_Y(X'). \] 
On the other hand, Blanchard's lemma (see e.g.
\cite[Prop.~4.2.1]{BSU}) yields a homomorphism
$\Aut(X')^0 \to \Aut(X)^0$, which restricts to a homomorphism 
\[ g_* : \Aut_Y(X')^0 \longrightarrow \Aut_K(X)^0. \]
Since $g$ is birational, $g_*$ is the inverse of the
restriction to neutral components 
$u^0 : \Aut_K(X)^0 \to \Aut_Y(X')^0$. 
In particular, the image of $u$ contains $\Aut_Y(X')^0$,
and hence is closed in $\Aut_Y(X')$.
\end{proof}

\medskip \noindent
{\bf Step 3.} In view of Step 2, we now consider a contraction
$f : X \to Y$, i.e., a morphism of normal projective varieties
such that $f_*(\cO_X) = \cO_Y$. We assume that the algebraic group
$G := \Aut_Y(X)^0$ is linear, and $f$ induces an isomorphism
$k(Y) \simeq K$. To prove Theorem \ref{thm:rel}, it suffices 
to show that $\Aut_Y(X)$ is a linear algebraic group. 

It suffices in turn to construct a big line bundle 
on $X$ which admits an $\Aut_Y(X)$-linearization
(recall that a line bundle $L$ on a projective variety 
$Z$ is big if there exists $c > 0$ such that 
$\dim H^0(Z,L^{\otimes n}) > c \, n^{\dim(Z)}$
for $n \gg 0$; see e.g. \cite[Lem.~2.60]{KM} for further
characterizations of bigness). Indeed,
$\Aut_Y(X)$ is closed in $\Aut(X)$ (as follows from 
Lemma \ref{lem:closed}), and hence in the stabilizer 
$\Aut(X,[L])$ of the isomorphism class of $L$ for any 
$\Aut_Y(X)$-linearized line bundle $L$. If in addition
$L$ is big, then $\Aut(X,[L])$ is a linear algebraic group 
in view of \cite[Lem.~2.3]{Br19}. The desired line bundle 
will be constructed in Step 6, after further preparations.

Denote by $\eta$ the generic point of $Y$ and by $X_{\eta}$
the generic fiber of $f$. Then $X_{\eta}$ is a projective
scheme over $k(\eta) = K$, equipped with an action of
the $K$-group scheme $G_K$ (since $G$ acts on $X$
by relative automorphisms). 
Likewise, the geometric generic fiber
$X_{\bar{\eta}}$ is a projective scheme over 
$k(\bar{\eta}) = \bar{K}$, equipped with an action of
$G_{\bar{K}}$ (a linear algebraic group over $\bar{K}$).

\begin{lemma}\label{lem:generic}
With the above assumptions, $X_{\eta}$ is normal
and geometrically integral. Moreover, $X_{\bar{\eta}}$
is almost homogeneous under $G_{\bar{\eta}}$ and
we have $k(X_{\bar{\eta}}) = k(X) \otimes_K \bar{K}$.
\end{lemma}

\begin{proof}
By \cite[Lem.~IV.1.5]{Springer}, the extension $k(X)/K$ 
is separable. As a consequence, the ring 
$k(X) \otimes_K \bar{K}$ is reduced (see e.g. 
\cite[10.42.5]{SP}). So $X_{\eta}$ is geometrically 
reduced. As $K$ is algebraically closed in $k(X)$, 
the spectrum of $k(X) \otimes_K \bar{K}$ is irreducible
(see e.g. \cite[10.46.8]{SP}). Hence $X_{\eta}$ is
geometrically irreducible and the field of rational functions
$\bar{K}(X_{\bar{\eta}})$ equals  $k(X) \otimes_K \bar{K}$. 
Thus,
\[ \bar{K}(X_{\bar{\eta}})^{G_{\bar{\eta}}} = 
(k(X) \otimes_K \bar{K})^{G_{\bar{\eta}}} \subset
(k(X) \otimes_K \bar{K})^G = \bar{K}, \]
where $G$ is identified with its group of $k$-rational points.
So $\bar{K}(X_{\bar{\eta}})^{G_{\bar{\eta}}} = k(\bar{\eta})$.
By Rosenlicht's theorem, it follows that $X_{\bar{\eta}}$
is almost homogeneous under $G_{\bar{\eta}}$.

It remains to show that $X_{\eta}$ is normal. Consider
an open affine subset $V$ of $Y$ and an open
affine cover $(U_i)$ of $f^{-1}(V)$. Then the 
$(U_i)_{\eta}$ form an open affine cover of $X_{\eta}$
and $\cO((U_i)_{\eta}) = \cO(U_i) \otimes_{\cO(V)} k(\eta)$
is a localization of $\cO(U_i)$, hence a normal domain.
\end{proof}

\begin{remark}\label{rem:generic}
The generic fiber $X_{\eta}$ is geometrically normal if
$\charac(k) = 0$. But this fails if $\charac(k) = p > 0$, as
shown by the following example: let $G = \bG_a$ act
on $X = \bP^2$ via 
$t \cdot [x : y : z] = [x + t y + t^{2p} z :  y : z]$.
Then $K = k(\frac{y}{z})$ and the $G$-orbit closure
of $[x : y : z]$ is singular at $\infty$ whenever $z \neq 0$. 
Likewise, the geometric generic fiber 
$X \times_{\Spec(K)}\Spec(\bar{K})$ is a singular curve.
\end{remark}

\medskip \noindent
{\bf Step 4.} We may further assume that $X(k(\eta))$
contains a point $x_0$ such that the orbit 
$G_{\bar{\eta}} \cdot x_0$ is open in $X_{\bar{\eta}}$.

\begin{proof}
Denote by $K^s$ the separable closure of $K$ in $\bar{K}$. 
Then $X(K^s)$ is dense in $X_{\bar{\eta}} = X_{\bar{K}}$,
since the latter is integral (Lemma \ref{lem:generic}). 
So we may find $x_0 \in X(K^s)$ such that 
$G_{\bar{\eta}} \cdot x_0$ is open in $X_{\bar{\eta}}$.
Then $x_0 \in X(K')$ for some finite Galois extension $K'/K$. 

Denote by $\Gamma$ the Galois group of $K'/K$, by
$Y'$ the normalization of $Y$ in $K'$, and by $X'$ 
the normalization of $X$ in $k(X) \otimes_K K'$ 
(the latter is a field as a consequence of Lemma 
\ref{lem:generic}). Then we have a commutative square
\begin{equation}\label{eqn:square} \xymatrix{
X' \ar[r]^-{f'} \ar[d]_g & Y' \ar[d]_h\\
X \ar[r]^-f & Y, \\
} \end{equation}
where $g,h$ are the quotients by $\Gamma$. Denoting by $\eta'$ 
the generic point of $Y'$, this yields a commutative square
\begin{equation}\label{eqn:generic} \xymatrix{
X'_{\eta'} \ar[r] \ar[d] & \eta' \ar[d] \\
X_{\eta} \ar[r] & \eta, \\
} \end{equation}
where the vertical arrows are quotients by $\Gamma$ again, and
$x_0$ is a section of the top horizontal arrow. Since the right 
vertical arrow is a $\Gamma$-torsor, so is the left vertical arrow
and the square (\ref{eqn:generic}) is cartesian. Therefore,
\[ X'_{\overline{\eta'}} = X'_{\eta'} \times_{\eta'} \overline{\eta'} 
= (X_{\eta} \times_{\eta} \eta') \times_{\eta'} \overline{\eta'} 
= X_{\eta} \times_{\eta} \overline{\eta'}
= X_{\eta} \times_{\eta} \bar{\eta} = X_{\bar{\eta}}. \]
In particular, $X'_{\eta'}$ is normal and geometrically integral
(Lemma \ref{lem:generic} again).

It follows that $f'_*(\cO_{X'}) =\cO_{Y'}$ by an argument of Stein
factorization as in Step 2. More specifically, $f'$ is the composition
\[ X' \longrightarrow {\rm Spec} \, f'_*(\cO_{X'}) =: Z' 
\stackrel{g'}{\longrightarrow} Y', \]
where $Z'$ is a variety and $g'$ is finite. Moreover, 
$Z'_{\eta'}$ is geometrically integral (since so is $X'_{\eta'}$)
and the finite morphism $Z'_{\eta'} \to \eta'$ has a section 
(since so does $X'_{\eta'} \to \eta'$).
Thus, $g'$ is an isomorphism at $\eta'$. In view of the normality 
of $Y'$ and Zariski's Main Theorem, we conclude that $g'$ is 
an isomorphism.

Next, we construct an isomorphism 
\begin{equation}\label{eqn:iso} 
 \Aut^{\Gamma}(X') \stackrel{\simeq}{\longrightarrow} \Aut(X), 
\end{equation}
where the right-hand side denotes the subgroup of $\Gamma$-invariants 
in $\Aut(X')$. For any scheme $S$, (\ref{eqn:square}) yields 
a commutative square
\[ \xymatrix{
X' \times S \ar[r] \ar[d] & Y' \times S \ar[d] \\
X \times S \ar[r] & Y \times S, \\
} \]
where the vertical arrows are again quotients by $\Gamma$.
Since these quotients are categorical (see e.g.
\cite[\S 12, Thm.~1]{Mumford}), every $\Gamma$-equivariant
automorphism of $X' \times S$ over $S$ induces 
an automorphism of $X \times S$ over $S$. This yields 
a morphism of (abstract) groups
\[ \Aut^{\Gamma}_S(X' \times S) \longrightarrow \Aut_S(X \times S) \]
which is clearly functorial in $S$. Thus, we obtain a morphism
of automorphism group schemes
$u : \Aut^{\Gamma}_{X'} \to \Aut_X$ with an obvious
notation. The induced morphism of Lie algebras is the natural map
$\Der^{\Gamma}(\cO_{X'}) \to \Der(\cO_X)$
with an obvious notation again. The kernel of this map is
contained in $\Der_{k(X)}(k(X'))$, and hence is zero since
$k(X')$ is separable algebraic over $k(X)$. 
Restricting $u$ to the reduced subscheme of $\Aut^{\Gamma}_{X'}$
yields a homomorphism 
\[ v : \Aut^{\Gamma}(X') \longrightarrow \Aut(X) \]
such that the induced morphism of Lie algebras is injective; 
thus, $v$ is \'etale. Its set-theoretic kernel is contained in 
$\Aut^{\Gamma}_{k(X)}(k(X'))$, and hence is trivial by 
Galois theory.  So $v$ is a closed immersion. 
Moreover, for any $g \in \Aut(X)$, we have a 
$\Gamma$-equivariant automorphism $g \otimes \id$ of
$k(X) \otimes_K  K' = k(X')$, which stabilizes the normalization
of $\cO_X$ and hence yields a lift of $g$ in
$\Aut^{\Gamma}(X')$. So $v$ is surjective on $k$-rational
points. This yields the desired isomorphism (\ref{eqn:iso}).

By construction, $v$ restricts to an isomorphism
\begin{equation}\label{eqn:isorel}
\Aut^{\Gamma}_{Y'}(X') \stackrel{\simeq}{\longrightarrow} 
\Aut_Y(X).
\end{equation} 
In particular, if $\Aut_{Y'}(X')$ is linear algebraic, then 
so is $\Aut_Y(X)$.

Since $G = \Aut_Y(X)^0$ has an open orbit in the general 
fibers of $f$, it follows that the connected algebraic group
$G' := \Aut_{Y'}(X')^0$ has an open orbit in the
general fibers of $f'$. As a consequence, $K' = k(X')^{G'}$.
Likewise, the orbit $G'_{\overline{\eta'}} \cdot x_0$ is open 
in $X'_{\overline{\eta'}}$. 

To complete the proof, it remains to show that if
$X'$ admits a big line bundle $L'$ which is 
$\Aut_{Y'}(X')$-linearized, then $X$ admits 
a big line bundle $L$ which is $\Aut_Y(X)$-linearized. 
This will follow from a norm argument. More specifically, let  
$M := \otimes_{\gamma \in \Gamma} \, \gamma^*(L')$.
Then $M$ is a big line bundle on $X'$ and is
$\Aut_{Y'}(X')^{\Gamma}$-linearized. Moreover, 
$M = g^*(L)$ for a line bundle $L$ on $X$ 
(the norm of $M$, see \cite[II.6.5]{EGA}); we have
$L = g_*(M)^{\Gamma}$. Thus, $L$ is 
$\Aut_Y(X)$-linearized in view of the isomorphism
(\ref{eqn:isorel}). Furthermore, we have 
\[ H^0(X',M^{\otimes n}) = H^0(X', g^*(L)^{\otimes n})
= H^0(X,L^{\otimes n} \otimes g_*(\cO_{X'})) \]
for any integer $n$. Since $g_*(\cO_{X'})^{\Gamma} = \cO_X$,
it follows that  
\[ H^0(X',M^{\otimes n})^{\Gamma} = H^0(X, L^{\otimes n}). \] 
Thus, the section ring
\[ R(X,L):= \bigoplus_{n = 0}^{\infty} \, H^0(X,L^{\otimes n}) \]
satisfies $R(X,L) = R(X',M)^{\Gamma}$. As a consequence,
the fraction fields of the domains $R(X,L)$ and $R(X',M)$ have 
the same transcendence degree over $k$. Since $M$ is big,
it follows that $L$ is big as well. 
\end{proof}

\medskip \noindent
{\bf Step 5.} Let $d : = \dim(X) - \dim(Y)$; this is the
maximal dimension of the $G$-orbits in $X$.
Let $X_0$ denote the set of $x \in X$ such that
the orbit $G \cdot x$ has dimension $d$; then $X_0$ 
is an open $G$-stable subset of $X$. Since $\Aut_Y(X)$
normalizes $G$, it stabilizes $X_0$ as well.

We first consider the case where $\charac(k) = 0$. Denote 
by $\fg$ the Lie algebra of $G$; then $\Aut_Y(X)$ acts 
on $\fg$ via its conjugation action on $G$. Consider 
the Tits morphism (see \cite{Haboush,HO})
\[ \tau_0 : X_0 \longrightarrow \Grass^d(\fg), 
\quad x \longmapsto \fg_x, \]
where $\fg_x \subset \fg$ denotes the Lie algebra of the
stabilizer of $x$, and $\Grass^d(\fg)$ stands for the Grassmannian 
of subspaces of $\fg$ of codimension $d$. Then $\tau_0$ 
is equivariant for the natural actions of $\Aut_Y(X)$.
We view $\tau_0$ as a rational map
$X \dasharrow \Grass^d(\fg)$, and denote by $X'$ the 
normalization of its graph. Then $X'$ is equipped with morphisms 
\[ \varphi : X' \longrightarrow X,  \quad 
\tau' : X' \longrightarrow \Grass^d(\fg) \] 
such that $\varphi$ restricts to an isomorphism above $X_0$, 
and $\tau' = \tau_0 \circ \varphi$. Moreover, the action of 
$\Aut_Y(X)$ on $X$ lifts to an action on $X'$ such that 
$\varphi$ and $\tau'$ are equivariant. Arguing as at the end 
of Step 2, one may check that $G \simeq \Aut_Y(X')^0$ and 
the image of $\Aut_Y(X)$ in $\Aut_Y(X')$ is closed. 
Since $f \circ \varphi$ is a contraction, this yields a reduction
to the case where $\tau_0$ extends to a morphism 
\begin{equation}\label{eqn:tits0}
\tau : X \longrightarrow \Grass^d(\fg).
\end{equation}

Next, we consider the case where $\charac(k) = p > 0$. 
For any integer $n > 0$, we then have the $n$th iterated 
Frobenius homomorphism
\[ F^n_{G/k} : G \longrightarrow G^{(p^n)} \]
and its kernel $G_n$. Recall that $G_n$ is an infinitesimal
group scheme, characteristic in $G$; moreover, its formation
commutes with base change, see e.g. 
\cite[VIIA.4.1]{SGA3}. (Note that $G_1$ is the infinitesimal
group scheme of height $1$ corresponding to the Lie algebra of
$G$). For the $G_n$-action on $X$, the stabilizer $\Stab_{G_n}$ 
is a subgroup scheme of $G_{n,X} = G_n \times X$,
stable by the diagonal action of $\Aut_Y(X)$.
Since $X$ is integral, it has a largest open subset $U = U_n$
over which $\Stab_{G_n}$ is flat, or equivalently locally free.
Then $U$ is stable by $\Aut_Y(X)$; moreover, since 
$\cO_{\Stab_{G_n}}$ is a quotient of $\cO(G_n) \otimes \cO_X$,
we obtain a morphism
\[ \tau_n : U \to \Grass^m(\cO(G_n)), \quad 
x \longmapsto \Stab_{G_n}(x) \]
for some $m \geq 1$. By construction, $\tau_n$ is 
$\Aut_Y(X)$-equivariant. As above, we may reduce to the
case where $\tau_n$ extends to a morphism 
\begin{equation}\label{eqn:titsp}
\tau : X \longrightarrow \Grass^m (\cO(G_n)).
\end{equation}

\medskip \noindent
{\bf Step 6.} We first record a general observation that will be used 
repeatedly. Given a morphism of projective varieties $f : X \to Y$,
we say that a line bundle on $X$ is $f$-\emph{big}
if its pull-back to the generic fiber $X_{\eta}$ is big 
(this notion turns out to be equivalent to that of 
\cite[Def.~3.22]{KM}).

\begin{lemma}\label{lem:fbig}
Let $f: X \to Y$ be a morphism of projective varieties, 
$L$ a line bundle on $X$, and $M$ a line bundle on $Y$.
Assume that $L$ is $f$-big and $M$ is ample. Then 
$L \otimes f^*(M^{\otimes n})$ is big for $n \gg 0$.
\end{lemma}

\begin{proof}
We adapt the argument of \cite[Lem.~2.4]{CCP}. 
Denote by $\eta$ the generic point of $Y$. 
Let $A$ be an ample line bundle on $X$, and $m$ a positive 
integer. Then $f_*(L^{\otimes m} \otimes A^{-1})$ is 
a coherent sheaf on $Y$, and 
$f_*(L^{\otimes m} \otimes A^{-1})_{\eta} = 
H^0(X_{\eta}, L^{\otimes m} \otimes A^{-1})$
(see e.g. \cite[III.9.4]{Hartshorne}). 
Since $L$ is big on $X_{\eta}$, it follows that
$f_*(L^{\otimes m} \otimes A^{-1})_{\eta} \neq 0$
for $m \gg 0$ by arguing as in the proof of 
\cite[Lem.~2.60]{KM}.
As $M$ is ample on $Y$, we thus have 
$H^0(Y, f_*(L^{\otimes m} \otimes A^{-1})
\otimes M^{\otimes n}) \neq 0$ for $n \gg m \gg 0$. 
Equivalently,
$H^0(X, L^{\otimes m} \otimes A^{-1} 
\otimes f^*(M^{\otimes n})) \neq 0$. So 
$L^{\otimes m} \otimes f^*(M^{\otimes n}) = A \otimes E$
for some effective line bundle $E$ on $X$. Thus, taking 
$n$ to be a large multiple of $m$ yields 
the statement in view of \cite[Lem.~2.60]{KM} again.
\end{proof}

Next, we return to the setting of Step 5, and consider 
again the open $\Aut_Y(X)$-stable subset $X_0 \subset X$ 
consisting of $G$-orbits of maximal dimension $d$. Replacing
$X$ with the normalized blow-up of $X \setminus X_0$,
we may assume that there is an effective Cartier divisor
$\Delta$ on $X$ such that $X \setminus \Supp(\Delta)$
consists of $G$-orbits of dimension $d$, and 
$\cO_X(\Delta)$ is $\Aut_Y(X)$-linearized; in particular,
$\Supp(\Delta)$ is $\Aut_Y(X)$-stable.

Let $\tau$ be the morphism as in (\ref{eqn:tits0})
(if $\charac(k) = 0$) or in (\ref{eqn:titsp}) (if $\charac(k) = p$).
Denote by $M$ the Pl\"ucker line bundle on the Grassmannian
and let $L := \tau^*(M)$. Then $L$ is a line bundle on $X$,
equipped with an $\Aut_Y(X)$-linearization. So
$L^{\otimes m}(\Delta)$ is $\Aut_Y(X)$-linearized as well
for any integer $m$.

\begin{lemma}\label{lem:big}
With the above notation, the line bundle 
$L^{\otimes m}(\Delta)_{\eta}$ on $X_{\eta}$ is big
for $m \gg 0$ (and for $n \gg 0$ if $\charac(k) > 0$).
\end{lemma}

\begin{proof}
Choose $x_0 \in X(k(\eta))$ such that 
$G_{\bar{\eta}} \cdot x_0$ is open in $X_{\bar{\eta}}$
(Step 4); then $G_{\bar{\eta}} \cdot x_0 = (X_0)_{\bar{\eta}}$.
In particular, $X_{\eta}$ is a normal projective variety,
almost homogeneous under $G_{\eta}$. This is the setting of 
Theorem 1 in \cite{Br19} (except that $k(\eta)$ need not be
algebraically closed), and the desired assertion may be obtained 
by arguing as in the last paragraph of the proof of that theorem. 
We provide an alternative argument based on Lemma \ref{lem:fbig}. 

Denote by $H$ the stabilizer of $x_0$ in $G_{\eta}$ 
and by $\fh$ its Lie algebra. If $\charac(k) = 0$, 
then we have the equality of normalizers 
$N_{G_{\eta}}(H^0) = N_{G_{\eta}}(\fh)$.
If $\charac(k) > 0$, then 
$N_{G_{\eta}}(H^0) = N_{G_{\eta}}(H_n)$
for $n \gg 0$, in view of \cite[Lem.~3.1]{Br19}.

Let $Z := \tau(X)$ and $z_0 := \tau(x_0)$, so that 
$z_0 = \fh$ if $\charac(k) = 0$ and $z_0 = H_n$
if $\charac(k) > 0$. Then $Z_{\eta}$ is almost homogeneous 
under $G_{\eta}$ and its open orbit is 
$G_{\eta} \cdot z_0 \simeq G_{\eta}/N_{G_{\eta}}(H^0)$. 
Moreover, $\tau$ restricts to an affine morphism
$f : G_{\eta}/H \to G_{\eta}/N_{G_{\eta}}(H^0)$ 
by \cite[Lem.~3.1]{Br19} again. Thus,
$X_{\eta} \setminus \Supp(\Delta_{\eta})$ is affine
over $G_{\eta} \cdot z_0$. Denoting by $\zeta$ the 
generic point of $Z$ (or equivalently of $Z_{\eta})$, 
it follows that 
$X_{\zeta} \setminus \Supp(\Delta_{\zeta})$ is affine.
As a consequence, $\Delta$ is $\tau$-big (see e.g. 
\cite[Lem.~2.4]{Br19}). Since $M$ is ample on $Z$, 
we conclude by Lemma \ref{lem:fbig}.
\end{proof}

By Lemmas \ref{lem:fbig} and \ref{lem:big}, $X$ admits 
a big line bundle which is $\Aut_Y(X)$-linearized. As seen 
in Step 3, this completes the proof of Theorem \ref{thm:rel}.

\begin{remark}\label{rem:char0}
Assume that $\charac(k) = 0$ and consider
$f : X \to Y$ as in Step 2. Choose a desingularization
$\psi : Y' \to Y$ and denote by $X'$ the normalization
of the irreducible component of $X \times_Y Y'$ which
dominates $Y$. By arguing as at the end of the proof
of Step 2, one shows that $\Aut_Y(X)$ is isomorphic
to a closed subgroup of $\Aut_{Y'}(X')$. So we may
assume that $Y$ is smooth.

We may further assume that $X$ is smooth, in view of
the existence of a canonical desingularization
(see \cite[Thm.~3.36]{Kollar}; such a desingularization 
is $\Aut(X)$-equivariant by the argument of 
\cite[Prop.~3.9.1]{Kollar}). Then the generic fiber
$X_{\eta}$ is smooth (by generic smoothness, see e.g.
\cite[III.10.7]{Hartshorne}); also, $X_{\bar{\eta}}$
is almost homogeneous under $G_{\bar{\eta}}$ in view
of Lemma \ref{lem:generic}. By \cite[Thm.~1.2]{FZ},
it follows that the anticanonical class
$-K_{X_{\bar{\eta}}}$ is big, hence so is $-K_{X_{\eta}}$. 
Also, $\cO_X(-K_X)$ is $\Aut(X)$-linearized.
Combining this with Lemmas \ref{lem:big} and \ref{lem:fbig}, 
this yields a shorter proof of Theorem \ref{thm:rel},
in characteristic zero again.
\end{remark}

\section{Proofs of Theorem \ref{thm:full} and Corollary \ref{cor:surface}}
\label{sec:full}

Like that of Theorem \ref{thm:rel}, the proof of Theorem \ref{thm:full} 
goes through a succession of reduction steps.

\medskip \noindent
{\bf Step 1.} Let $Y$ be the smooth projective curve
such that $K = k(Y)$. Then we may assume that the 
inclusion $k(Y) \subset k(X)$ comes from a contraction
$f : X \to Y$, and $G = \Aut_Y(X)^0$.

\begin{proof}
The invariant field $k(X)^{G^0}$ is algebraic over $K$, 
and therefore has transcendence degree $1$ over $k$. 
Thus, we may replace $G$ with $G^0$, and hence assume 
that $G$ is connected.

Denote by $f : X \dasharrow Y$ the rational
map associated with the inclusion $k(Y) \subset k(X)$. 
Arguing as in Step 2 of Section \ref{sec:rel}, 
we may further assume that $f$ is a morphism; 
then it is a contraction. 

By Lemma \ref{lem:neutral}, the group 
$\Aut_Y(X)^0 = \Aut_K(X)^0$ is linear; thus, we may 
also assume that $G = \Aut_Y(X)^0$.
\end{proof}

\medskip \noindent
{\bf Step 2.} We may further assume that 
$X$ is not almost homogeneous under $\tG := \Aut(X)^0$. 

\begin{proof}
By Blanchard's lemma again (see e.g. \cite[Prop.~4.2.1]{BSU}), 
the action of $\tG$ on $X$ induces a unique action
on $Y$ such that $f$ is equivariant. This yields an exact 
sequence of algebraic groups
\[ 1 \longrightarrow \tG \cap \Aut_Y(X)
 \longrightarrow \tG \longrightarrow \Aut(Y)^0, \]
where $\tG \cap \Aut_Y(X)$ has reduced neutral component
$G$, and hence is affine.

If $X$ is almost homogeneous under $\tG$, then
so is $Y$. Thus, either $Y \simeq \bP^1$
or $Y$ is an elliptic curve. In the former case,
we claim that every rational map from $X$
to an abelian variety is constant. Indeed, any
such map $\varphi : X \dasharrow A$ is $G$-invariant.
Since $k(X)^G = k(\bP^1)$, it follows that
$\varphi$ factors through a rational map 
$\bP^1 \dasharrow A$, implying the claim.

By this claim, the open orbit of $\tG$ in $X$
has a trivial Albanese variety. Arguing as in 
the proof of Lemma \ref{lem:neutral}, 
it follows that $\tG$ is linear. Thus,
so is $\Aut(X)$ in view of \cite[Thm.~1]{Br19}.

On the other hand, if $Y$ is an elliptic curve,
then $f$ is the Albanese morphism of $X$, since 
the latter morphism is $G$-invariant. Thus, there is
a unique action of $\Aut(X)$ on $Y$ such that 
$f$ is equivariant. This yields an exact sequence
of group schemes, locally of finite type
\begin{equation}\label{eqn:aut}
1 \longrightarrow  N \longrightarrow 
\Aut(X) \stackrel{f_*}{\longrightarrow} \Aut(Y),
\end{equation}
where the reduced subscheme of $N$ is $\Aut_Y(X)$.
In view of  Theorem \ref{thm:rel}, it follows that
$N$ is an affine algebraic group. 
Also, the image of $f_*$ contains $Y = \Aut(Y)^0$,
since $X$ is almost homogeneous under $\tG$.
As $\Aut(Y)$ is a non-linear algebraic group, 
we conclude that so is $\Aut(X)$.
\end{proof}

\medskip \noindent
{\bf Step 3.} We may further assume that 
$G = \tG = \Aut(X)^0$, and $\Aut(X)$ acts on $Y$ 
so that we still have the exact sequence (\ref{eqn:aut});
in addition, $Y \simeq \bP^1$ or $Y$ is an elliptic 
curve. 

\begin{proof}
By Rosenlicht's theorem again,
the subfield $k(X)^{\tG} \subset K$ has transcendence
degree $1$ over $K$. As $k(X)^{\tG}$ is algebraically
closed in $k(X)$ (Lemma \ref{lem:closed}), it follows 
that $k(X)^{\tG} = K$. So $\tG$ is linear in view of 
Lemma \ref{lem:neutral}.

Therefore, we may assume that $G = \Aut(X)^0$. 
Then $\Aut(X)$ stabilizes $K$, and hence acts on $Y$;
this yields (\ref{eqn:aut}).

Denote by $g$ the genus of $Y$. If $g \geq 2$ then 
$\Aut(Y)$ is finite, and hence $\Aut(X)$ is 
a linear algebraic group. So we may further assume 
that $g \leq 1$.
\end{proof}

\medskip \noindent
{\bf Step 4.} We may further assume that $X$ has 
an $\Aut(X)$-linearized line bundle $L$ which is $f$-big 
and $f$-globally generated, i.e., the adjunction map 
\[ u : f^* f_* (L) \longrightarrow L \] 
is surjective.

\begin{proof}
Arguing as in Steps 4, 5 and 6 of Section \ref{sec:rel}, 
we may assume that $X$ has an $f$-big line bundle 
$L$ which is $\Aut(X)$-linearized; moreover,
$H^0(X,L) \neq 0$. Then $f_*(L)$ is a coherent,
torsion-free sheaf on $Y$, and hence is locally
free. Thus, so is $f^* f_*(L)$; moreover, $u$
is a morphism of $\Aut(X)$-linearized sheaves.
Also, $u \neq 0$ as $L$ has non-zero global
sections. Therefore, $u$ yields a non-zero
morphism $L^{-1} \otimes f^* f_*(L) \to \cO_X$.
Its image is the ideal sheaf of a closed subscheme
$Z \subsetneq X$, stable by $\Aut(X)$. 

Denote by 
\[ \pi: X' \longrightarrow X \] 
the normalization of the blow-up of $Z$ in $X$ and let 
\[ f' := f \circ \pi : X' \longrightarrow Y. \]
Then the action of $\Aut(X)$ on $X$ lifts to a unique 
action on $X'$; moreover, $\pi^*(L)$ is 
$\Aut(X)$-linearized and we have
$f'^* f'_* \pi^*(L) = \pi^* f^* f_*(L)$.
The image of the adjunction map 
\[ u': f'^* f'_* \pi^*(L) \longrightarrow \pi^*(L) \] 
generates an invertible subsheaf $L' \subset \pi^*(L)$,
as proved (in essence) in 
\cite[Ex.~II.7.17.3]{Hartshorne}. Note that $L'$ is 
$\Aut(X)$-linearized and $f'$-globally generated.

We claim that $L'$ is $f'$-big, possibly
after replacing $L$ with a positive tensor
power $L^{\otimes m}$, and $L'$ with the
associated subsheaf 
$L'_m \subset \pi^*(L^{\otimes m})$. 
Indeed, denoting by $\eta$ the generic point
of $Y$, we have
\[ f_*(L^{\otimes m})_{\eta} = 
H^0(X_{\eta}, L^{\otimes m}) \]
by \cite[III.9.4]{Hartshorne}, and likewise,
\[ f'_* \pi^*(L^{\otimes m})_{\eta} = 
H^0(X'_{\eta}, \pi^*(L^{\otimes m})) =
H^0(X_{\eta}, L^{\otimes m}). \]
Thus, $\pi^*(L)$ is $f'$-big. Moreover,
denoting by $i : X_{\eta} \to X$ and
$i' : X'_{\eta} \to X'$ the inclusions,
we obtain
\[ 
i'^* f'^* f'_* \pi^*(L^{\otimes m}) =
f'_* \pi^*(L^{\otimes m})_{\eta} 
\otimes \cO_{X'_{\eta}} = 
H^0(X'_{\eta}, \pi^*(L^{\otimes m}))
\otimes \cO_{X'_{\eta}}.
\]
Since $L'_m$ is the subsheaf of 
$\pi^*(L^{\otimes m})$ generated by the image 
of $f'^* f'_* \pi^*(L^{\otimes m})$, 
we see that $i'^*(L'_m)$ is generated by
the image of the natural map
\[ H^0(X_{\eta}, L^{\otimes m}) 
\otimes \cO_{X'_{\eta}} 
\longrightarrow i'^* \pi^*(L^{\otimes m}). \] 
As $i'^* \pi^*(L^{\otimes m})$ is big, this
proves the claim.

This claim combined with \cite[Lem.~2.1]{Br19}
yields the desired reduction.
\end{proof}

\medskip \noindent
{\bf Step 5.} Choose an ample line bundle $M$ on $Y$. 
Then $L \otimes f^*(M^{\otimes n})$ is big for
$n \gg 0$ (Lemma \ref{lem:fbig}). We claim
that $L \otimes f^*(M^{\otimes n})$ is also globally 
generated for $n \gg 0$. 

Indeed, since $M$ is ample, the sheaf
$f_*(L) \otimes M^{\otimes n}$ is globally generated 
for $n \gg 0$. Equivalently, the evaluation map
\[ H^0(Y,f_*(L) \otimes M^{\otimes n}) 
\otimes \cO_Y \to  
f_*(L) \otimes M^{\otimes n} \]
is surjective. As $f$ is flat, the induced map
\[ H^0(Y,f_*(L) \otimes M^{\otimes n}) 
\otimes \cO_X \to  
f^*(f_*(L) \otimes M^{\otimes n}) \]
is surjective as well. Since 
$H^0(Y,f_*(L) \otimes M^{\otimes n}) = 
H^0(X,f^*(f_*(L) \otimes M^{\otimes n}))$,
it follows that the sheaf $f^*(f_*(L) \otimes M^{\otimes n})$
is globally generated. Thus, so is its
quotient $L \otimes f^*(M^{\otimes n})$,
proving the claim.

\medskip \noindent
We may now complete the proof of Theorem \ref{thm:full}.
Observe that $\Aut(X)$ fixes the numerical
equivalence class of $f^*(M^{\otimes n})$
(since $\Aut(Y)$ fixes that of $M^{\otimes n}$)
and hence that of 
$L \otimes f^*(M^{\otimes n})$. Thus, $\Aut(X)$
is an algebraic group in view of \cite[Lem.~2.2]{Br19}.

If $Y \simeq \bP^1$, then $\Aut(X)$ fixes the isomorphism 
class of $f^* \cO_{\bP^1}(1)$, since $\Aut(Y) = \PGL_2$
fixes that of $\cO_{\bP^1}(1)$. So $\Aut(X)$ is linear
algebraic by \cite[Lem.~2.3]{Br19}. Otherwise,
$Y$ is an elliptic curve with function field $K$.

\medskip \noindent
{\bf Proof of Corollary \ref{cor:surface}.} 
By assumption, there exists a non-trivial connected linear 
algebraic subgroup $G \subset \Aut(X)$. Replacing $G$ 
with a Borel subgroup, we may assume that it is solvable. 
Then $k(X)$ is a purely transcendental field extension of 
$K := k(X)^G$ in view of \cite[Thm.~1]{Popov}.
Also, recall that $K$ is algebraically closed in $k(X)$
(Lemma \ref{lem:closed}). Since $K \neq k(X)$, it follows
that either $K = k$, or $K$ has transcendence degree
$1$ over $k$. In the former case, $\Aut(X)$ is linear
algebraic by \cite[Thm.~1]{Br19}. In the latter case,
Theorem \ref{thm:full} yields that $\Aut(X)$ is algebraic, 
and linear unless $K$ is the function field of an elliptic 
curve $Y$; then $k(X) \simeq K(t) \simeq k(Y \times \bP^1)$.

\medskip \noindent
{\bf Acknowledgments.} This work was prompted by 
a minicourse on automorphism groups of algebraic varieties
that I gave at the summer school 
``Visions of Algebraic Groups'', held at St Petersburg,
26--30 August 2019. I thank the organizers of this event 
for their invitation, and the participants, especially 
Ivan Arzhantsev, for stimulating discussions. Also, thanks 
to the anonymous referee for valuable comments and suggestions.

\bibliographystyle{amsalpha}

\end{document}